\documentclass[12pt]{article}

\usepackage{amsthm}
\usepackage{amsmath}
\usepackage{amssymb}
\usepackage{amsmath, amsthm, bbm}
\usepackage{framed, fullpage}
\usepackage[backref,colorlinks,citecolor=blue,bookmarks=true]{hyperref}

\newtheorem{theo}{Theorem}

\newtheorem{lemma}{Lemma}[section]
\newtheorem{coro}[lemma]{Corollary}
\newtheorem{definition}[lemma]{Definition}

\newtheorem{claim}[lemma]{Claim}

\newcommand{\E}{\mathbb{E}}
\newcommand{\ZZ}{\mathbb{Z}}
\newcommand{\Z}{\ZZ}

\newcommand{\ith}{i\text{-th}}

\newcommand{\s}[1]{\left\lvert #1 \right\rvert}
\newcommand{\floor}[1]{\left\lfloor{#1}\right\rfloor}
\newcommand{\ceil}[1]{\left\lceil #1 \right\rceil}

\DeclareMathOperator{\twr}{twr}

\newcommand{\eps}{\epsilon}
\newcommand{\e}{\epsilon}

\newcommand{\sm}{\setminus}

\newcommand{\sprod}[2]{\langle{#1},{#2}\rangle}
\newcommand{\ip}[1]{\langle{#1}\rangle}

\date{}

\title{An Improved Lower Bound for Arithmetic Regularity}
\author{
Kaave Hosseini\thanks{CSE department, UC San Diego. Email: \texttt{skhossei@cse.ucsd.edu}. Supported by NSF award 1350481.}
\and
Shachar Lovett\thanks{CSE department, UC San Diego. Email: \texttt{slovett@cse.ucsd.edu}. Supported by NSF award 1350481.}
\and
Guy Moshkovitz\thanks{School of Mathematics, Tel Aviv University, Tel Aviv, Israel 69978.  Email: {\tt guymosko@tau.ac.il}. Supported in part by ISF grant 224/11.}
\and
Asaf Shapira\thanks{School of Mathematics, Tel Aviv University, Tel Aviv, Israel 69978. Email: {\tt asafico@tau.ac.il}. Supported in part by ISF Grant 224/11 and a Marie-Curie CIG Grant 303320.}
}

\begin{document}
\maketitle
\begin{abstract}
The arithmetic regularity lemma due to Green [GAFA 2005] is an analogue of the famous Szemer{\'e}di regularity lemma in graph theory. It shows that for any abelian group
$G$ and any bounded function $f:G \to [0,1]$, there exists a subgroup $H \le G$ of bounded index such that, when restricted to most cosets of $H$, the function $f$ is pseudorandom in the sense that all its nontrivial Fourier coefficients are small. Quantitatively, if one wishes to obtain that for $1-\eps$ fraction
of the cosets, the nontrivial Fourier coefficients are bounded by $\eps$, then Green shows that $|G/H|$ is bounded by a tower of twos of height $1/\eps^3$. He also
gives an example showing that a tower of height $\Omega(\log 1/\eps)$ is necessary. Here, we give an improved example, showing that a tower of height $\Omega(1/\eps)$ is necessary.
\end{abstract}

\section{Introduction}

As an analogue of Szemer\'edi's regularity lemma in graph theory~\cite{Szemeredi78},
Green~\cite{Green05} proposed an arithmetic regularity lemma for abelian groups.
Given an abelian group $G$ and a bounded function $f:G \to [0,1]$, Green showed that one can find a subgroup $H \le G$ of bounded index, such that when restricted to most cosets of $H$, the function $f$ is pseudorandom in the sense that all of its nontrivial Fourier coefficients are small. Quantitatively, the index of $H$ in $G$ is bounded by a tower of twos of height polynomial in the error parameter. The goal of this note is to provide an example showing that these bounds are essentially tight. This strengthens a previous example due to Green~\cite{Green05} which shows that a tower of height logarithmic in the error parameter is necessary; and makes the lower bounds in the arithmetic case analogous to these obtained in the graph case~\cite{Gowers97}.

We restrict our attention in this paper to the group $G=\ZZ_2^n$, and note that our construction can be generalized to groups of bounded torsion in an obvious way. We first make some basic definitions.
Let $A$ be an affine subspace (that is, a translation of a vector subspace) of $\ZZ_2^n$ and let $f:A\to[0,1]$ be a function.
The Fourier coefficient of $f$ associated with $\eta \in \ZZ_2^n$ is
$$\widehat{f}(\eta) = \frac{1}{\s{A}}\sum_{x\in A} f(x)(-1)^{\sprod{x}{\eta}} = \E_{x\in A} [f(x)(-1)^{\sprod{x}{\eta}}] \;.$$
Any subspace $H \le \ZZ_2^n$ naturally determines a partition of $\ZZ_2^n$ into affine subspaces
$$\ZZ_2^n/H = \{ H+g \,:\, g \in \ZZ_2^n\} \;.$$
The number $\s{\ZZ_2^n/H}= 2^{n-\dim H}$ of translations  is called the \emph{index} of $H$.

\subsection{Arithmetic regularity and the main result}

For an affine subspace $A=H+g$ of $\ZZ_2^n$,
where $H \le \ZZ_2^n$ and $g \in \ZZ_2^n$,
we say that a function $f:A\to[0,1]$ is \emph{$\e$-regular} if
all its nontrivial Fourier coefficients are bounded by $\e$, that is,
$$\max_{\eta \notin H^\perp} \big\lvert \widehat{f}(\eta) \big\rvert \le \e \;.$$
Note that a trivial Fourier coefficient $\eta \in H^{\perp}$ satisfies $\lvert \widehat{f}(\eta) \rvert = \s{\E_{x\in A} f(x)}$.
Henceforth, for any $f:\ZZ_2^n\to[0,1]$ we denote by $f|_A:A\to[0,1]$ the restriction of $f$ to $A$.

\begin{definition}[$\e$-regular subspace]
Let $f:\ZZ_2^n\to[0,1]$.
A subspace $H \le \ZZ_2^n$ is \emph{$\e$-regular}
for %with respect to
$f$ if $f|_A$ is $\e$-regular for at least $(1-\e)\s{\ZZ_2^n/H}$ translations $A$ of $H$.
\end{definition}

% translation between ours and Green's definition
%$$\widehat{f_H^{+g}}(\eta) = \sum _{h \in H} f(h+g)(-1)^{\sprod{h}{\eta}} = \sum _{y+g \in H} f(y)(-1)^{\sprod{y+g}{\eta}} = (-1)^{\sprod{g}{\eta}}\sum _{y \in H+g} f(y)(-1)^{\sprod{y}{\eta}} = (-1)^{\sprod{g}{\eta}}\widehat{f|_{H+g}}(\eta)$$

Green~\cite{Green05} proved that any %function $f:\ZZ_2^n \to [0,1]$
bounded function
has an $\e$-regular subspace $H$ of bounded index, that is, whose index
depends only on $\e$ (equivalently, $H$ has bounded codimension). In the following, 
$\twr(h)$ is a tower of twos of height $h$; formally, $\twr(h):=2^{\twr(h-1)}$ for a positive integer $h$, and $\twr(0)=1$.

\begin{theo}[Arithmetic regularity lemma in $\ZZ_2^n$, Theorem~2.1 in~\cite{Green05}]\label{theo:Green}
For every $0<\e<\frac12$ there is $M(\e)$ such that every function $f:\ZZ_2^n \to [0,1]$ has an $\e$-regular subspace of index at most $M(\e)$.
Moreover, $M(\e) \le \twr(\lceil 1/\e^3 \rceil)$.
\end{theo}

A lower bound on $M(\e)$ of about $\twr(\log_2(1/\e))$
%of $M(\e) \ge \twr(\frac12 \log_2(1/\e))$
was given in the same paper~\cite{Green05}, following the lines of Gowers' lower bound on the order of $\e$-regular partitions of graphs~\cite{Gowers97}.
While Green's lower bound implies that $M(\e)$ indeed has a tower-type growth,
it is still quite far from the upper bound in Theorem~\ref{theo:Green}.

Our main result here nearly closes the gap between the lower and upper bounds on $M(\e)$, showing that $M(\e)$ is a tower of twos of height at least linear in $1/\e$.
Our construction follows the same initial setup as in~\cite{Green05}, but will diverge from that point on. Our proof is inspired by the recent simplified lower bound proof for the graph regularity lemma in~\cite{MoSh13} by
a subset of the authors.

\begin{theo}\label{theo:main}
For every $\e>0$
it holds that $M(\e) \ge \twr(\floor{1/16\e})$.
%Moreover, there are Boolean functions $f:\ZZ_2^n\to\{0,1\}$ achieving this bound.
\end{theo}

\subsection{A variant of Theorem~\ref{theo:main} for binary functions}

One can also deduce from Theorem~\ref{theo:main} a similar bound for $\e$-regular \emph{sets}, that is, for binary functions $f:\ZZ_2^n\to\{0,1\}$.
For this, all we need is the following easy probabilistic argument.

\begin{claim}\label{claim:weightedSet}
Let $\tau>0$ and $f:\ZZ_2^n\to[0,1]$. There exists a binary function $S:\ZZ_2^n\to\{0,1\}$ satisfying, for every affine subspace $A$ of $\ZZ_2^n$ of size $\s{A} \ge 4n^2/\tau^2$ and
any vector $\eta\in \ZZ_2^n$, that
$$
\big\lvert \widehat{S|_A}(\eta)-\widehat{f|_A}(\eta) \big\rvert \le \tau.
$$
%$\s{A} \ge nlgn/\tau^2$.
\end{claim}

\begin{proof}
Choose $S:\ZZ_2^n\to\{0,1\}$ randomly by setting $S(x)=1$ with probability $f(x)$, independently for each $x\in \ZZ_2^n$,
%, where for each vector $x\in \ZZ_2^n$ we set $R(x)$ independently
Let $A,\eta$ be as in the statement. The random variable $$\widehat{S|_A}(\eta) = \frac{1}{\s{A}} \sum_{x\in A} S(x)(-1)^{\sprod{x}{\eta}}$$
is an average of $\s{A}$ mutually independent random variables taking values in $[-1,1]$, and its expectation is $\widehat{f|_A}(\eta)$. By Hoeffding's bound, the probability that
$\big\rvert \widehat{S|_A}(\eta)-\widehat{f|_A}(\eta) \big\rvert > \tau$ is smaller than
$$2\exp(-\tau^2\s{A}/2) \le 2^{-2n^2 + 1} \;.$$
The number of affine subspaces over $\Z_2^n$ can be trivially bounded by $2^{n^2}$, the number of sequences of $n$ vectors in $\Z_2^n$. Hence, the number of pairs $(A,\eta)$ is bounded by $2^{n^2+n}$.
The claim follows by the union bound.
\end{proof}

Applying Claim~\ref{claim:weightedSet} with $\tau=\e/2$ (say) implies that if $f:\ZZ_2^n\to[0,1]$ has no $\e$-regular subspace of index smaller than $\twr(\floor{1/16\e})$ then, provided $n$ is sufficiently large in terms of $\e$, there is $S:\ZZ_2^n\to\{0,1\}$ that has no $\e/2$-regular
%$(\e-\tau)$-regular
subspace of index smaller than $\twr(\floor{1/16\e})$.

\section{Proof of Theorem~\ref{theo:main}}

\subsection{The Construction}%\label{subsec:Construction}

To construct a function witnessing the lower bound in Theorem~\ref{theo:main} we will use pseudo-random spanning sets.
%\footnote{For an explicit construction (with somewhat worst constants) one can use the small bias generators of~\cite{NaorNa93}.}.

\begin{claim}\label{claim:quasirandom}
Let $V$ be a vector space over $\ZZ_2$ of dimension $d$.
Then there is a set of $8d$ nonzero vectors in $V$ such that any subset of $\frac34$ of them spans $V$.
\end{claim}
\begin{proof}
Choose random vectors $v_1,\ldots,v_{8d}\in V\sm\{0\}$ independently and uniformly. Let $U$ be a subspace of $V$ of dimension $d-1$.
The probability that a given $v_i$ lies in $U$
% a given subspace $U\sub V$ of codimension $1$
is at most $\frac12$. By Chernoff's bound, the probability that more than $6d$ of our vectors $v_i$ lie in $U$ is smaller than $\exp(-2(2d)^2/8d) = \exp(-d)$. By the union bound, the probability that there exists a subspace $U$ of dimension $d-1$ for which the above holds is at most $2^d\exp(-d)< 1$. This completes the proof.
\end{proof}

We now describe a function $f:\ZZ_2^n\to[0,1]$ which, as we will later prove, has no $\e$-regular subspace of small index.
%Henceforth, set $\d = 10\e$, $s=\floor{1/\delta}$.
Henceforth set $s=\floor{1/16\e}$.
Furthermore, let $d_i$ be the following sequence of integers of tower-type growth:
$$d_{i+1} = \begin{cases}
	2^{D_i} & \text{if } i=1,2,3\\
	2^{D_i-3} & \text{if } i > 3
\end{cases} \qquad \text{where } D_i=\sum_{j=1}^{i} d_j \text{ and } D_0=0 \;.$$
Note that the first values of $d_i$ for $i\ge 1$ are $1,2,8,2^8,2^{264},$ etc.,\ and it is not hard to see that $d_i \ge \twr(i-1)$ for every $i\ge 1$.
Set $n=D_s$ $(\ge \twr(s-1))$.
For $x\in\ZZ_2^n$, partition its coordinates into $s$ blocks of sizes $d_1,\ldots,d_{s}$, and identify $x=(x^1,\ldots,x^s) \in \ZZ_2^{d_1+\cdots+d_s} = \ZZ_2^{n}$.

Let $1 \le i \le s$.
Bijectively associate with each $v\in \ZZ_2^{D_{i-1}}=\ZZ_2^{d_1+\cdots+d_{i-1}}$ a nonzero vector $\xi_i(v) \in \ZZ_2^{d_{i}}$ such that the set of vectors $\{\xi_i(v) : v\in \ZZ_2^{D_{i-1}} \}$ has the property that any subset of $\frac34$ of its elements spans $\ZZ_2^{d_i}$. The existence of such a set, which is a subset of size $2^{D_{i-1}}$ in a vector space of dimension $d_{i}$, follows from
Claim~\ref{claim:quasirandom} when $i > 3$, since then $2^{D_{i-1}} = 8d_{i}$.
When $i \le 3$ the existence of such a set is trivial since $\ceil{(3/4)i} = i$, hence any basis would do (and we take
$2^{D_{i-1}}=d_{i}$).
With a slight abuse of notation, if $x\in\Z_2^n$ we write $\xi_i(x)$ for $\xi_i((x^1,\ldots,x^{i-1}))$.

We define our function $f:\ZZ_2^n\to[0,1]$ as
$$f(x) = \frac{\s{\{1 \le i \le s \,:\, \sprod{x^i}{\xi_i(x)} = 0 \}}}{s} \;.$$
The following is our main technical lemma, from which Theorem~\ref{theo:main} immediately follows.
\begin{lemma}\label{lemma:main}
The only $\e$-regular subspace for $f$ is the zero subspace $\{0\}$.
\end{lemma}
\begin{proof}[Proof of Theorem~\ref{theo:main}]
The index of $\{0\}$ is $\s{\ZZ_2^n/\{0\}} = 2^n \ge \twr(s) = \twr(\floor{1/16\e})$.
\end{proof}

\subsection{Proof of Lemma~\ref{lemma:main}}
%\begin{proof}[Proof of Lemma~\ref{lemma:main}]

Let $H$ be an $\e$-regular subspace for $f$, and assume towards contradiction that $H \neq \{0\}$. Let $1 \le i\le s$ be minimal such that there exists $v \in H$ for which $v^i \neq 0$.
For any $g \in \ZZ_2^n$, let
$$\gamma_g = (0,\ldots,0,\xi_i(g),0,\ldots,0) \in \ZZ_2^n $$ where only the $\ith$ component is nonzero.
%where $g^j=0$ for $j\neq i$, and $g^i=\z_i(g^1,\ldots,g^{i-1})$.
%We will show that for a noticeable fraction of translations $H+g$ it holds that
%$f|_{H+g}$ is not $\e$-regular, and more specifically,
We will show that for more than an $\e$-fraction of the translations $H+g$ of $H$ it holds
that $\gamma_g \notin H^{\perp}$ and
$$
\widehat{f|_{H+g}}(\gamma_g) > \e \;.
%\s{\widehat{f|_{H+g}}(\gamma_g)} > \e \;.
$$
This will imply that $H$ is not $\e$-regular for $f$, thus completing the proof.

First, we argue that $\gamma_g \notin H^{\perp}$ for a noticeable fraction of $g \in \Z_2^n$. We henceforth let $B=\{g \in \Z_2^n: \gamma_g \in H^{\perp}\}$ be the set of "bad" elements.
\begin{claim}\label{claim:Bsmall}
$\s{B} \le \frac34 \s{\Z_2^n}$.
\end{claim}

\begin{proof}
If $g \in B$ then $\ip{\xi_i(g),v^i}=0$. Hence, $\{\xi_i(g): g \in B\}$ does not span $\Z_2^{d_{i}}$. By the construction of $\xi_i$, this means that $\{(g^1,\ldots,g^{i-1}): g \in B\}$ accounts to at most $\frac34$ of the elements in $\Z_2^{D_{i-1}}$, and hence $\s{B} \le \frac34\s{\Z_2^n}$.
\end{proof}

Next, we argue that typically $\widehat{f|_{H+g}}(\gamma_g)$ is large. Let $W \le \Z_2^n$ be the subspace spanned by the last $s-i$ blocks, that is, $W=\{w \in \Z_2^n: w^1=\ldots=w^i=0\}$. Note that for any $g \in \Z_2^n, w \in W$ we have $\gamma_{g+w}=\gamma_g$. In particular, $g+w \in B$ if and only if $g \in B$.

\begin{claim}\label{claim:fourier}
Fix $g \in \Z_2^n$ such that $\gamma_g \notin H^{\perp}$. Then
$$
\E_{w \in W}\left[\widehat{f|_{H+g+w}}(\gamma_{g})\right] = \frac{1}{2s} \;.
$$
\end{claim}

\begin{proof}
Write $f(x) = \frac{1}{s} \sum_{j=1}^s B_j(x)$ where $B_j(x):\Z_2^n\to\{0,1\}$ is the characteristic function for the set of vectors $x$ satisfying $\sprod{x^j}{\xi_j(x)} = 0$. Hence, for any affine subspace $A$ in $\Z_2^n$,
\begin{equation}\label{eq:linearity}
\widehat{f|_A}(\gamma_g) = \frac{1}{s} \sum_{j=1}^s \widehat{B_j|_A}(\gamma_g)\;.
\end{equation}
Set $A=H+g+w$ for an arbitrary $w \in W$.
We next analyze the Fourier coefficient $\widehat{B_j|_A}(\gamma_g)$ for each $j \le i$, and note that in these cases we have $\xi_j(x) = \xi_j(g)$ for any $x \in A$.
First, if $j<i$ then for every $x\in A$ we have $x^j = g^j$, which implies that $B_j|_{A}$ is constant. Since a nontrivial Fourier coefficient of a constant function equals $0$, we have
\begin{equation}\label{eq:j<i}
\widehat{B_j|_A}(\gamma_g) = 0, \qquad \forall j<i.
\end{equation}
Next, for $j=i$,
write $B_i|_A(x) = \frac12((-1)^{\sprod{x^i}{\xi_i(x)}}+1)$.
Since $\ip{x,\gamma_g}=\ip{x^i,\xi_i(x)}$, we have
\begin{equation}\label{eq:j=i}
\widehat{B_i|_A}(\gamma_g) = \E_{x \in A} \left[ \frac12((-1)^{\ip{x^i,\xi_i(x)}}+1) \cdot (-1)^{\ip{x^i,\xi_i(x)}} \right] =
\E_{x \in A} \left[ B_i(x) \right] = \frac12 \;.
\end{equation}
Finally, for $j>i$ we average over all $w \in W$.
%The average over $w \in W$ will play a role in cancellation for $j>i$.
Let $H+W$ be the subspace spanned by $H,W$.
Writing $B_j(x) = \frac12((-1)^{\sprod{x^j}{\xi_j(x)}}+1)$,
the average Fourier coefficient is
$$
\E_{w \in W} \E_{x \in H+g+w} \left[ B_j(x) (-1)^{\ip{x^i,\xi_i(x)}} \right] = \frac12\E_{x \in H+W+g} \left[ (-1)^{\ip{x^i,\xi_i(g)}+\ip{x^j,\xi_j(x)}} \right].
$$
Note that for every fixing of $x^1,\ldots,x^{j-1}$, we have that $x^j$ is uniformly distributed in $\Z_2^{d_j}$ (due to $W$), and that $(-1)^{\ip{x^i,\xi_i(g)}}$ is constant. Since $\xi_j(x)\neq 0$, we conclude that
\begin{equation}\label{eq:j>i}
\E_{w \in W} \left[ \widehat{B_j|_{H+g+w}}(\gamma_g) \right] = 0, \qquad \forall j>i.
\end{equation}
The proof now follows by substituting~(\ref{eq:j<i}),~(\ref{eq:j=i}) and~(\ref{eq:j>i}) into~(\ref{eq:linearity}).
\end{proof}

Since $\widehat{f|_{H+g+w}}(\gamma_g) \le 1$, we infer (via a simple averaging argument) the following corollary.

\begin{coro}\label{cor:fourier}
Fix $g \in \Z_2^n$ such that $\gamma_g \notin H^{\perp}$. Then for more than $1/4s$ fraction of the elements $w \in W$,
$$
\widehat{f|_{H+g+w}}(\gamma_g) > \frac{1}{4s} \;.
$$
\end{coro}

We are now ready to conclude the proof of Lemma~\ref{lemma:main}. Partition $\Z_2^n$ into translations of $W$, and recall that $\gamma_g$ depends just on the translation $g+W$. By Claim~\ref{claim:Bsmall}, for at least $\frac14$ of the translations, $\gamma_g \notin H^{\perp}$. By Corollary~\ref{cor:fourier}, in each such translation, more than $1/4s$-fraction of the elements $g+w$ satisfy
$\widehat{f|_{H+g+w}}(\gamma_g) > 1/4s$.
Since $1/16s \ge \e$, the subspace $H$ cannot be $\e$-regular for $f$.

%\end{proof}

\end{document}